\documentclass[12pt]{article}

\usepackage{graphicx} % Required for inserting images
\usepackage{hyperref}
\usepackage{amsmath}
\usepackage{cleveref}
\usepackage{amsthm}
\usepackage{amssymb}
\usepackage{enumitem}
\usepackage{comment}
\theoremstyle{plain}% default
\newtheorem{theorem}{Theorem}[section]

\newtheorem{proposition}[theorem]{Proposition}
\newtheorem{corollary}[theorem]{Corollary}

\theoremstyle{definition}
\newtheorem{definition}[theorem]{Definition}
\newtheorem{conjecture}[theorem]{Conjecture}
\newtheorem{example}[theorem]{Example}

\usepackage{amsfonts,braket,mathtools}
\usepackage{lipsum}
\usepackage{authblk}
\usepackage{blindtext}

\usepackage{tikz-cd}
\usepackage{tikz}
\usetikzlibrary{decorations}
\usetikzlibrary{arrows}
\tikzstyle{v} = [circle, draw, inner sep=2pt, minimum size=3pt, fill=black]
\tikzstyle{l} = [rectangle, draw, rounded corners]

\DeclareMathOperator{\rank}{rank}

\title{A generalization of $q$-deformation of graphic arrangements to simplicial complexes}
\author[1]{Tongyu Nian\thanks{u487971i@ecs.osaka-u.ac.jp}}
\affil[1]{Department of Mathematics, The University of Osaka, Toyonaka, Osaka 560-0043, Japan.}
\date{}
\usepackage[backend=bibtex,giveninits=true,minalphanames=4,maxalphanames=999,style=alphabetic]{biblatex}
\renewbibmacro{in:}{}
\begin{document}

\maketitle

\section{Introduction}

The purpose of this thesis is to introduce two new kinds of hyperplane arrangements, inspired by the graphic arrangements and $q$-deformations of graphic arrangements \cite{nian2024qdeformationchromaticpolynomialsgraphical}. 

In \cite{nian2024qdeformationchromaticpolynomialsgraphical}, the authors observe a mysterious similarity between the braid arrangement and the arrangement of all hyperplanes in a vector space over the finite field of order $q$. These two arrangements are defined by the determinants of the Vandermonde and the Moore matrices, respectively. They are transformed to each other by replacing a natural number $n$ with $q^n$ (q-deformation). Following the similarity, the authors introduce $q$-deformation of graphical arrangements as certain subarrangements of the arrangement of all hyperplanes over $\mathbb{F}_q$. Then they show that many invariants of the $q$-deformation behave as $q$-deformation of invariants of the graphical arrangements including the characteristic polynomials, the Stirling numbers of the second kind, freeness of both objects, the corresponding exponents and bases of logarithmic vector fields, etc.

In this thesis, the author extends the definition of $q$-deformation to simplicial complexes, with the conjecture in \cite{nian2024qdeformationchromaticpolynomialsgraphical}. The author also investigates a special case called graphic monomial arrangement, including the characteristic polynomials and freeness with a further extension to fields with primitive roots.

The organization of this thesis is as follows. In Section \ref{prelim}, there is a brief introduction to the theory of hyperplane arrangements. In Section \ref{section2}, the author extends the definition of $q$-deformation to simplicial complexes and extend Conjecture \ref{q-deformation conjceture} to a general case. In section \ref{section1dim}, the author investigates a special class when the simplicial complex is $1$-dimensional, giving the characteristic polynomial and showing how it is related to graphic arrangement. In section \ref{section3}, the author extends the definition in section \ref{section1dim} to any field with root of unity and states results available in this category.

\section{Preliminary}\label{prelim}
In this section, we give a brief introduction to the theory of hyperplane arrangements, including basic definitions and notations, graphic arrangements, $q$-deformations and key tools of freeness. For more details of hyperplane arrangements, see \cite{orlik1992arrangements}.

A \textbf{hyperplane arrangement}, in short, arrangement, is a collection of affine subspaces of codimension $1$, called hyperplanes, in a vector space $\mathbb{K}^\ell$. If all hyperplanes in an arrangement are linear subspaces, the arrangement is called \textbf{central}.

Since each hyperplane corresponds to the set of zeros of some linear polynomial, the \textbf{defining polynomial} $Q(\mathcal{A})$ is to take the product of such polynomials corresponding to each hyperplane in the arrangement.

Let $\mathcal{A}=\{H_1,H_2,\cdots, H_n\}$ be a hyperplane arrangement, $H_I=\bigcap_{i\in I}H_i$ where $I\subseteq [n]$, then the \textbf{characteristic polynomial} of $\mathcal{A}$ is defined by
    $$\chi(\mathcal{A}, t)\coloneqq\sum_{I\subseteq [n],H_I\neq \emptyset}(-1)^{\left|I\right|}t^{\dim{H_I}}.$$

An arrangement $\mathcal{A}$ is \textbf{supersolvable} if and only if there exists a chain of subarrangements called a filtration
\begin{align*}
\varnothing=\mathcal{A}_{0} \subseteq \mathcal{A}_{1} \subseteq \mathcal{A}_{2} \subseteq \dots \subseteq \mathcal{A}_{\ell} = \mathcal{A}
\end{align*}
such that, for each $i \in [\ell]$, $\rank \mathcal{A}_{i} = i$  and for any distinct hyperplanes $H,H^{\prime} \in \mathcal{A}_{i}\setminus\mathcal{A}_{i-1}$ there exists $H^{\prime\prime}\in\mathcal{A}_{i-1}$ such that $H \cap H^{\prime} \subseteq H^{\prime\prime}$ \cite{bjorner1990hyperplane-dcg}. 
When $\mathcal{A}$ is supersolvable with the previous filtration, the characteristic polynomial can be written as the product of linear factors, with coefficients derived from the filtration, say 
\begin{align*}
\chi(\mathcal{A},t) = \prod_{i=1}^{\ell}(t-|\mathcal{A}_{i}\setminus\mathcal{A}_{i-1}|). 
\end{align*}
In other words, an arrangement $\mathcal{A}$ is supersolvable if and only if its intersection lattice is supersolvable. Here the intersection lattice is a poset on nonempty intersections of members in the arrangement with partial order given by reverse inclusion. A lattice is supersolvable if there exists a maximal chain of modular elements in the lattice\cite{stanley1972supersolvable}.

Let $\mathcal{A}=\{H_1,H_2,\cdots,H_n\}$ be a central arrangement in $V=\mathbb{K}^\ell$. Denote the polynomial ring $\mathbb{K}[x_1,x_2,\cdots,x_\ell]$ by $S$. Define the \textbf{derivation module} $\mathrm{Der}_V$ by 
    \begin{align*}
        \mathrm{Der}_V\coloneqq \{\theta:S\rightarrow S\mid \theta \text{ is a $\mathbb{K}$-linear map and }\theta(fg)=\theta(f)g+\theta(g)f\}
    \end{align*}
    where $f$ and $g$ are polynomials in $S$. Also denote a submodule
    $$D(\mathcal{A})\coloneqq\{\theta\in \mathrm{Der}_V\mid \theta(Q(\mathcal{A}))\in (Q(\mathcal{A}))\}.$$
 If $D(\mathcal{A})$ is a free module with generators $\theta_1, \theta_2,\cdots,\theta_\ell$, with $\deg \theta_i=e_i$, which means
    $$D(\mathcal{A}) = S\theta_1 \oplus S\theta_2 \oplus \cdots \oplus S\theta_\ell,$$
    we say $\mathcal{A}$ is a \textbf{free} arrangement with basis $(\theta_1, \theta_2,\cdots,\theta_\ell)$ and exponents $\exp(\mathcal{A})=(e_1,e_2,\cdots,e_\ell)$. For a free arrangement, we have the following theorem.

    \begin{theorem}[\cite{terao1981generalized}]
        If $\mathcal{A}\subset \mathbb{K}^\ell$ is free with exponents $(e_1,e_2,\cdots, e_\ell)$, then
    $\chi(\mathcal{A}, t)=\prod_{i=1}^\ell (t-e_i)$.
    \end{theorem}

A key tool to check if a given set of derivations is a proper free basis is the Saito's criterion\cite{saito1980theory}.

\begin{theorem}[Saito's criterion]
    Let $\mathcal{A}$ be a central hyperplane arrangement in $V=\mathbb{K}^\ell$. Then $\theta_1,\theta_2,\cdots,\theta_\ell\in D(\mathcal{A})$ is a free basis if and only if $\det(\theta_i (x_j))=cQ(\mathcal{A})$, where $c$ is a nonzero constant. If the derivatives are written in $\theta_i=\sum_{j=1}^\ell f_{ij}\partial_{x_j}$, then $\det(\theta_i (x_j))=\det(f_{ij})$, this determinant is called Saito's determinant.
\end{theorem}

One can also check the freeness of specific arrangements by the celebrated Terao's addition-deletion theorem.

\begin{theorem}[Terao \cite{terao1980arrangements}]
    Let $H \in \mathcal{A}$, $\mathcal{A}' \coloneqq \mathcal{A} \setminus {H}$ and $\mathcal{A}^H \coloneqq \{L \cap H \mid L \in \mathcal{A}'\}$. Then any two of the following imply the third:
    \begin{enumerate}
        \item $\mathcal{A}$ is free with $\exp(\mathcal{A}) = (d_1, d_2 , \cdots, d_{\ell-1}, d_{\ell})$.
        \item $\mathcal{A}'$ is free with $\exp(\mathcal{A}') = (d_1, d_2 , \cdots, d_{\ell-1}, d_{\ell}-1)$.
        \item $\mathcal{A}^H$ is free with $\exp(\mathcal{A}^H) = (d_1, d_2 , \cdots, d_{\ell-1})$.
    \end{enumerate}
\end{theorem}

Note that the subarrangement of a free arrangement can be non-free \cite{edelman1993counterexample}. For more general discussion on freeness, see also \cite{dimca2017hyperplane}. 

In general, supersolvability implies freeness but the inverse is false. We say that a central arrangement $\mathcal{A}$ is inductively free if it can be constructed by using only the addition theorem from the empty arrangement. Every supersolvable arrangement is inductively free \cite{jambu1984free-aim} and inductively free arrangement is also free. Conversely, we can find free but not inductively free arrangements and also inductively free but not supersolvable arrangements\cite{orlik1992arrangements}.

In this thesis, we will focus on graphic arrangement, the arrangement associated to a graph. 

\begin{definition}
    Let $G=([\ell], E)$ be a simple graph on $\ell$ vertices, the graphic arrangement $\mathcal{A}_G$ is defined by
    $$\mathcal{A}_G\coloneqq\{\ker{(x_i-x_j)}\mid (i,j)\in E\}.$$
    The defining polynomial of this arrangement is $$Q(\mathcal{A}_G)=\prod_{(i,j)\in E}(x_i-x_j).$$
\end{definition}

Graphic arrangement can also be seen as subarrangement of Coxeter arrangement of type $A$. The characteristic polynomial of graphic arrangement is equal to the chromatic polynomial of the graph.

In \cite{nian2024qdeformationchromaticpolynomialsgraphical}, the authors introduce another kind of arrangement on graphs called $q$-deformation, defined as follows.
\begin{definition}
Let $G=([\ell], E)$ be a simple graph on $\ell$ vertices, $C(G)$ denotes the set of cliques of the graph. The $q$-deformation of graphic arrangement $\mathcal{A}_G$ is defined by 
    \begin{align*}
\mathcal{A}_{G}^{q} \coloneqq \bigcup_{\{i_{1}, \dots, i_{r}\}\in C(G)}\ \Set{\ker  (a_{i_{1}}x_{i_{1}} +a_{i_2}x_{i_2} +\dots + a_{i_{r}}x_{i_{r}} )  \mid a_{i_j}\in \mathbb{F}_q^\times\text{ for all }j}, 
\end{align*}
\end{definition}

We call this arrangement ``$q$-deformation'' because a certain limit $q\to 1$ is expected to recover the chromatic polynomial as follows.
% {\color{red} (finish the sentence.)}

\begin{conjecture}[Nian-Tsujie-Uchiumi-Yoshinaga \cite{nian2024qdeformationchromaticpolynomialsgraphical}]\label{q-deformation conjceture}
The characteristic polynomial of the $q$-deformation $\chi(\mathcal{A}_G^q,t)$ is a polynomial in $q$ and $t$, such that
    \[
\lim_{q\to 1}\frac{\chi(\mathcal{A}_G^q, q^s)}{(q-1)^\ell}
=\chi(G, s),
\]
where $\chi(G, t)$ denotes the chromatic polynomial of the graph $G$.
\end{conjecture}

So $q$-deformation is a $q$-analog of graphic arrangement. It is checked that the $q$-deformations of graphic arrangements of chordal graphs and triangle-free graphs satisfy the conjecture.

Chordal graph is an important class of graphs related to supersolvability and freeness. Rose\cite{rose1970triangulated} proved that a graph is chordal if and only if there is a \textbf{perfect elimination ordering} of the vertices $(v_1,v_2,\cdots,v_\ell)$ such that for any $i$, any two distinct vertices adjacent to $v_i$ in $\{v_1,v_2,\cdots,v_i\}$ are also adjacent.

The following theorem summarizes the relationship among the chordality of a graph, the freeness and supersolvability of the corresponding graphic arrangement and $q$-deformation.
\begin{theorem}[Stanley (See {\cite[Theorem 3.3]{edelman1994free-mz}}), Dirac \cite{dirac1961rigid-aadmsduh}, Fulkerson-Gross \cite{fulkerson1965incidence-pjom}, Nian-Tsujie-Uchiumi-Yoshinaga \cite{nian2024qdeformationchromaticpolynomialsgraphical}]\label{graphical arrangement freeness}
The following are equivalent. 
\begin{enumerate}[label=(\arabic*)]
\item\label{PEO} $G$ has a perfect elimination ordering. 
\item $\mathcal{A}_{G}$ is supersolvable. 
\item $\mathcal{A}_{G}$ is free. 
\item\label{chordal} $G$ is chordal. 
\item\label{q SS} $\mathcal{A}_{G}^{q}$ is supersolvable. 
\item\label{q free} $\mathcal{A}_{G}^{q}$ is free. 
\end{enumerate}
\end{theorem}

The filtration $\varnothing=\mathcal{A}_{0} \subseteq \mathcal{A}_{1} \subseteq \mathcal{A}_{2} \subseteq \dots \subseteq \mathcal{A}_{\ell} = \mathcal{A}_G$ given by 
$$\mathcal{A}_{i}=\mathcal{A}_{G_i},$$
allows us to write the characteristic polynomial of graphic arrangement of any chordal graph into linear factors where $G_i$ is the subgraph induced by $\{v_1,v_2,\cdots,v_i\}$, or equivalently,
\begin{align*}
\chi(\mathcal{A}_{G}, t) = \prod_{i=1}^{\ell}(t-|N_{G_{i}}(v_{i})|), 
\end{align*}
where $N_{G_{i}}(v_{i})$ is the neighborhood of $v_{i}$ in $G_{i}$. 

For $q$-deformations of graphic arrangements of chordal graphs, we have
\begin{align*}
\chi(\mathcal{A}_{G}^{q}, t) = \prod_{i=1}^{\ell}\left(t-q^{|N_{G_{i}}(v_{i})|}\right).
\end{align*}

\section{Generalization to simplicial complexes}\label{section2}
A \textbf{simplicial complex} $\Delta=(V,\mathcal{F})$ consists of a vertex set $V$ and a simplicial face set $\mathcal{F}\subseteq 2^V$ such that for any face $F\in\mathcal{F}$, any subset of $F$ is also a face of this complex.
% \subsection{Definition}

Following the manner of $q$-deformation, we can extend the definition by replacing cliques by faces in simplicial complexes.

\begin{definition}
    Let $\Delta=([\ell], \mathcal{F})$ be a simplicial complex on $\ell$ vertices, the generalized $q$-deformation associated to $\Delta$ is defined by
    $$\mathcal{S}_\Delta^q \coloneqq \bigcup_{\{i_{1}, \dots, i_{r}\}\in \mathcal{F}}\ \{ \ker (a_{i_{1}}x_{i_{1}} + \dots + a_{i_{r}}x_{i_{r}}) \mid a_{i_j}\in \mathbb{F}_q^\times\text{ for all }j\}.$$
    % {\color{red} add $($ $)$ in the ker}
\end{definition}

In this sense, the set of cliques of a graph form a simplicial complex which is called the clique complex, denoted by $\mathcal{C}(G)$. The $q$-deformation $\mathcal{A}_G^q$ of a graph is equal to $\mathcal{S}_{\mathcal{C}(G)}^q$. 
It is also easy to see any simplicial complex is the intermediate complex between the $1$-skeleton of itself, say underlying graph and the clique complex of its underlying graph. So we extend the conjecture in \cite{nian2024qdeformationchromaticpolynomialsgraphical} to a more general case.
\begin{conjecture}\label{main conjecture}
    For any simplicial complex $\Delta=([\ell], \mathcal{F})$, the characteristic polynomial $\chi(\mathcal{S}_\Delta^q, t)$ is a polynomial in $q$ and $t$, such that
    \[
\lim_{q\to 1}\frac{\chi(\mathcal{S}_\Delta^q, q^s)}{(q-1)^\ell}
=\chi(G, s),
\]
where $\chi(G, t)$ denotes the chromatic polynomial of the underlying graph $G$.
% {\color{red} change $t$ to $s$}
\end{conjecture}

Following the same method of vector coloring in \cite{nian2024qdeformationchromaticpolynomialsgraphical}, we can prove the congruence between both sides of the previous equation when $t$ is a nonnegative integer $k$. 

\begin{proposition}
Let $q$ be a prime power and $k$ be a nonnegative integer, then
    \[
\frac{\chi(\mathcal{S}_\Delta^q, q^k)}{(q-1)^\ell}
\equiv\chi(G, k) \pmod{q-1}.
\]
\end{proposition}

There is no efficient way to compute the characteristic polynomial of this kind of arrangements other than using deletion-restriction formula repeatedly. However, we can relate the characteristic polynomial to the underlying graph for some special cases.

\begin{theorem}[$q$-deletion-contraction]
    \label{q-deletion-contraction}
    Let $\Delta=([\ell], \mathcal{F})$ be a simplicial complex. Suppose $e=\{i_1, i_2\}$ is an edge of its underlying graph. If $e$ is a maximal face in $\Delta$, then by letting $\Delta/e$ be the simplicial complex identifying $i_1$ and $i_2$ in each face of $\Delta$, we have \[
    \chi(\mathcal{S}_{\Delta}^{q}, t)=\chi(\mathcal{S}_{\Delta\setminus e}^{q}, t)-(q-1)\chi(\mathcal{S}_{\Delta/e}^{q}, t). 
    \]
\end{theorem}

This theorem can be seen as a $q$-analog of the deletion-contraction formula of the chromatic polynomial of a graph.

\begin{proof}
    Use deletion-restriction formula repeatedly on hyperplanes of the form $x_{i_1}+ax_{i_2}=0$.
\end{proof}

It is hard to check if an arrangement derived from general simplicial complex is free efficiently. For example, we have the following computation on subcomplexes of a simplex. 

\begin{example}
    Let $\Delta_{\ell,k}$ be the $(k-1)$-skeleton of a $(\ell-1)$-simplex. Then we have the following results.
    \begin{enumerate}
        \item $\mathcal{S}_{\Delta_{\ell,\ell-1}}^q$ is free. The characteristic polynomial is given by using deletion-restriction repeatedly:
        $$\chi(\mathcal{S}_{\Delta_{\ell,\ell-1}}^q,t)=(t-1)(t-q)\cdots(t-q^{\ell-2})(t-q^{\ell-1}+(q-1)^{\ell-1}),$$
        then $\chi(\mathcal{S}_{\Delta_{\ell,\ell-1}}^q,q^{\ell-2})=0$ hence it is free by using the theorem of Yoshinaga \cite[Theorem 11]{yoshinaga2007free-potjasams}.
        \item $\mathcal{S}_{\Delta_{5,3}}^q$ is not free. To be specific,
        \begin{align*}
            \chi(\mathcal{S}_{\Delta_{5,3}}^q,t)=& (t-1)(t-q)(t-q^2)\\
            &(t^2-(9q^2-11q+4)t+21q^4-54q^3+57q^2-29q+6),
        \end{align*}
        when $q=2$, we have
        \begin{align*}
            \chi(\mathcal{S}_{\Delta_{5,3}}^2,t) &= (t-1)(t-2)(t-4)(t^2-18t+80)\\
            &=(t-1)(t-2)(t-4)(t-8)(t-10);
        \end{align*}
        when $q=3$, we have
        \begin{align*}
            \chi(\mathcal{S}_{\Delta_{5,3}}^3,t) &= (t-1)(t-3)(t-9)(t^2-52t+675)\\
            &=(t-1)(t-3)(t-9)(t-25)(t-27);
        \end{align*}
        when $q=4$, we have
        \begin{align*}
            \chi(\mathcal{S}_{\Delta_{5,3}}^4,t) &= (t-1)(t-4)(t-16)(t^2-104t+2722),
        \end{align*}
        and it cannot be factorized into linear factors since $2722=2\times 1361$;

         when $q=5$, we have
         \begin{align*}
            \chi(\mathcal{S}_{\Delta_{5,3}}^5,t) &= (t-1)(t-5)(t-25)(t^2-174t+7661)
        \end{align*}
        and it cannot be factorized into linear factors since $7661=47\times 163$.
    \end{enumerate}
\end{example}

\section{One-dimensional case}\label{section1dim}

In this section, we restrict simplicial complexes to the case of dimension $1$ which means there are only $0$-faces and $1$-faces in the complex, corresponding to vertices and edges in the underlying graph. In other word, we can consider a simple graph $G$ as a $1$-dimensional simplicial complex, denoted by $\Delta_G$.

\begin{definition}\label{Definition of graphic monomial arrangement}
    Let $G=([\ell], E)$ be a simple graph on $\ell$ vertices and denote the corresponding $1$-dimensional simplicial complex by $\Delta_G$.
    The arrangement associated to $\Delta_G$ with underlying graph $G$ is defined by
    \begin{align*}
        \mathcal{S}_G^q &\coloneqq \mathcal{S}_{\Delta_G}^q \\
        &=\{\ker{(x_i+ax_j)}\mid (i,j)\in E, a\in \mathbb{F}_q\}.
    \end{align*}
   
\end{definition}

The defining equation of this arrangement is 
$$Q(\mathcal{S}_G^q)=\prod_{i=1}^\ell x_i\prod_{(i,j)\in E}(x_i^{q-1}-x_j^{q-1}).$$

Following Theorem \ref{q-deletion-contraction}, we can see that each edge in $1$-dimensional simplicial complex is maximal, hence $q$-deletion-contraction holds for any edge in the graph. We restate Theorem \ref{q-deletion-contraction} using $\mathcal{S}_G^q$.

\begin{theorem}\label{q-deletion-contraction of GMA}
    Let $G=([\ell], E)$ be a simple graph on $\ell$ vertices. If $e=\{i_1, i_2\}$ is an edge, then
    $$\chi(\mathcal{S}_{G}^{q}, t)=\chi(\mathcal{S}_{G\setminus e}^{q}, t)-(q-1)\chi(\mathcal{S}_{G/e}^{q}, t).$$
\end{theorem}

Note that for the empty graph $G_0 =([\ell], \emptyset)$ on $\ell$ vertices, the characteristic polynomial is $\chi(\mathcal{S}_{G_0}^q,t)=(t-1)^\ell$, while the chromatic polynomial of the empty graph on $\ell$ vertices is $t^\ell$, by induction on the number of edges in the graph. we have the following proposition.

\begin{proposition}\label{prop4.3}
    Let $G=([\ell], E)$ be a simple graph on $\ell$ vertices, $\chi(G, t)$ be chromatic polynomial of $G$, then
    $$\chi(\mathcal{S}_{G}^{q}, t)=(q-1)^\ell\chi(G,\frac{t-1}{q-1}).$$
    {In particular, $\chi(\mathcal{S}_{G}^{q}, t)$ is a polynomial in $q$ and $t$.}
\end{proposition}

This formula shows the connection with graphic arrangements. So this kind of arrangements plays an important role between the graphic arrangements and $q$-deformations of the graphic arrangements.

Following this formula, we can directly check that this kind of arrangement satisfies  Conjecture \ref{main conjecture}.

\begin{theorem}\label{3}
    Generalized $q$-deformation derived from any $1$-dimensional simplicial complex satisfies Conjecture \ref{main conjecture}. In other word, let $G$ be a graph on $\ell$ vertices. Then $\chi(\mathcal{S}_G^q,t)$ is polynomial in $q$ and $t$, and satisfies 
    \begin{align*}
        \lim_{q \to 1} \frac{\chi(\mathcal{S}_G^q,q^s)}{(q-1)^\ell}=\chi(G,s).
    \end{align*}
\end{theorem}
\begin{proof}
Combining the definition of generalized $q$-deformation of simplicial complex and Proposition \ref{prop4.3}, we can compute
    \begin{align*}
        \lim_{q \to 1} \frac{\chi(\mathcal{S}_G^q,q^s)}{(q-1)^\ell} &=
        \lim_{q\to 1}\frac{\chi(\mathcal{S}_{\Delta^1}^q, q^s)}{(q-1)^\ell}\\ &= \lim_{q\to 1}\frac{\chi(\mathcal{S}_{G(\Delta^1)}^q, q^s)}{(q-1)^\ell}\\
        &= \lim_{q\to 1}\frac{(q-1)^\ell\chi(G(\Delta^1), \frac{q^s-1}{q-1})}{(q-1)^\ell}\\
        &= \lim_{q\to 1}\chi(G(\Delta^1), \frac{q^s-1}{q-1})\\
        &= \chi(G(\Delta^1), s)\\
        &= \chi(G,s).
    \end{align*}
\end{proof}

We shall look at the following statement as a corollary that identifies a kind of graph satisfying Conjecture \ref{q-deformation conjceture}.

\begin{corollary}
    All triangle-free graphs $G$ satisfy Conjecture \ref{q-deformation conjceture}.
\end{corollary}

\begin{proof}
    Observe that the $q$-deformation of graphic arrangement of a triangle-free graph is $\mathcal{S}_G^q$ since there is no clique of more than $2$ vertices. Then use the Theorem \ref{3} on the clique complex of this graph.
\end{proof}

\section{Graphic monomial arrangements}\label{section3}

In this section, we extend the field in $q$-deformations to any field with roots of unity. We follow the language of graph theory and name the extended class of arrangements by graphic monomial arrangements, referring to the reflection arrangements of monomial groups in \cite{orlik1992arrangements} when we consider in a wider range of fields. In this section, let $r$ be a positive integer. We consider arrangements over the field $\mathbb{K}$ containing a primitive $r$-th root $z\in\mathbb{K}$ of $1$.

\begin{definition}
    Let $G=([\ell], E)$ be a graph. Define the \textbf{graphic monomial arrangement} $\mathcal{M}(G,r)$ by
    \[ \mathcal{M}(G,r)=\{\{x_i-z^k x_j=0\}\mid (i,j)\in E, 1\leq k\leq r\}\cup\{\{x_i=0\}\mid i= 1, \dots, \ell\}, \] and the \textbf{simplified graphic monomial arrangement} $\mathcal{M}^0(G,r)$ by
    \[ \mathcal{M}^0(G,r)=\{\{x_i-z^k x_j=0\}\mid (i,j)\in E, 1\leq k\leq r\}.\]
\end{definition}

This setting contains at least four important classes. 
\begin{itemize}
    \item[(a)] $\mathbb{K}=\mathbb{F}_q$, $r=q-1$;
    \item[(b)] $\mathbb{K}=\mathbb{C}$, $r$ is an arbitrary positive integer;
    \item[(c)] $\mathbb{K}=\mathbb{R}$, $r=2$;
    \item[(d)] Any field $\mathbb{K}$ with $r=1$.
\end{itemize}
 The case (a) corresponds to the arrangement in the previous section. In particular, $\mathcal{S}_G^q=\mathcal{M}(G, q-1)$. In the case (b), if $G=K_\ell$ is the complete graph, the defining equations of $\mathcal{M}^0(K_\ell, r)$ and $\mathcal{M}(K_\ell, r)$ are $\prod_{1\leq i<j\leq\ell}(x_i^r-x_j^r)$, and $x_1 x_2\cdots x_\ell\prod_{1\leq i<j\leq\ell}(x_i^r-x_j^r)$, respectively, which are known to be the reflection arrangements of monomial groups (monomial arrangements). Case (c), is related to a class of subarrangements of reflection arrangements of type $B$ and $D$. In particular, $\mathcal{M}(K_\ell, 2)$ and $\mathcal{M}^0(K_\ell, 2)$ are equal to reflection arrangements of type $B_\ell$ and $D_\ell$, respectively. 
Case (d) is the classical graphic arrangement. More precisely, $\mathcal{M}^0(G, 1)=\mathcal{A}_G$ for any graph $G$. 

There are also families of free multiarrangements closely related, proposed by Hoge, Mano, R{\"o}hrle and Stump \cite{hoge2019freeness} concerning about reflections, with an integral expression of free basis given by Feigin, Wang and Yoshinaga \cite{feigin2025integral}.

Similar to the case of dimension $1$, graphic monomial arrangements have $q$-deletion-contraction. We restate Theorem \ref{q-deletion-contraction of GMA} in the language of graphic monomial arrangements.

\begin{theorem}
    Let $G=([\ell], E)$ be a simple graph on $\ell$ vertices and $\mathbb{K}$ be a field with primitive $r$-root of unity. If $e=\{i_1, i_2\}$ is an edge, then the graphic monomial arrangement $\mathcal{M}(G,r)$ in $\mathbb{K}^\ell$ satisfies
    $$\chi(\mathcal{M}(G,r), t)=\chi(\mathcal{M}({G\setminus e},r), t)- r \chi(\mathcal{M}({G/e},r), t).$$
    
\end{theorem}

We also have

\begin{proposition}
    Let $G=([\ell], E)$ be a simple graph on $\ell$ vertices, $\chi(G, t)$ be chromatic polynomial of $G$, then
   $$\chi(\mathcal{M}(G,r),t)=r^\ell\chi(G,\frac{t-1}{r}).$$
    {In particular, $\chi(\mathcal{M}(G,r), t)$ is a polynomial in $r$ and $t$.}
\end{proposition}

It is easy to see if $G_0$ is the empty graph, $\mathcal{M}(G_0,r)$ and $\mathcal{M}^0(G_0,r)$ correspond to Boolean arrangement and empty arrangement respectively. We have 

$$\chi(\mathcal{M}(G_0, r), t)=\chi(\mathcal{M}^0(G_0, r), t-1).$$

When $r=1$, by induction on the number of edges of the graph, it leads to a well-known result: $\chi(\mathcal{M}^0(G, 1), t-1)=\chi(\mathcal{M}(G, 1), t)$ where $\mathcal{M}^0(G, 1)$ is graphic arrangement on $G$. For general $r>0$, we will see that $\mathcal{M}(G, r)$ inherits nice properties of $\mathcal{M}(G, 1)$. While simplified graphic monomial arrangement $\mathcal{M}^0(G,r)$ does not satisfy $q$-deletion-contraction for $r>1$, many results of $\mathcal{M}^0(G,1)$ cannot be extended to general $r$.

Next we will discuss the supersolvability and freeness of graphic monomial arrangements of chordal graphs. Following Theorem \ref{graphical arrangement freeness}, we aim to show

\begin{theorem}\label{chordality is freeness}
    Let $G=([\ell], E)$ be a simple graph on $\ell$ vertices, $\mathbb{K}$ be a field with $r$-th root of unity. Then $\mathcal{M}(G,r)$ is free if and only if $G$ is chordal.
\end{theorem}

Theorem \ref{graphical arrangement freeness} showed that all graphic arrangements of chordal graphs and $q$-deformations of them are supersolvable. The filtration for graphic arrangement of chordal graph is given by 
$$\mathcal{A}_i=\mathcal{A}_{G_i}$$
and the filtration for $q$-deformation is given by
$$\mathcal{A}_i=\mathcal{A}_{G_i}^q$$
where $G_i$ is the subgraph induced by $\{v_1,v_2,\cdots,v_i\}$ with perfect elimination ordering $(v_1,v_2,\cdots,v_\ell)$ of vertices.

One can check that the same construction also works for $\mathcal{M}(G,r)$ if $G$ is a chordal graph.

\begin{theorem}\label{filtration}
    Let $G$ be a chordal graph with perfect elimination ordering $(v_1,v_2,\cdots,v_\ell)$, then the arrangement $\mathcal{M}(G,r)$ is supersolvable with filtration
    $$\mathcal{A}_i=\mathcal{M}({G_i},r)$$
    where $G_i$ is the subgraph induced by $\{v_1,v_2,\cdots,v_i\}$.
\end{theorem}

\begin{proof}
    Directly check the definition. Choose a pair of hyperplanes in $\mathcal{A}_r\setminus\mathcal{A}_{r-1}$, say $x_r-z^{r_i}x_i=0$ and $x_r-z^{r_j}x_j=0$, where $v_i$ and $v_j$ are vertices adjacent to $v_r$ such that $i,j<r$. By chordality of the graph, $(v_i,v_j)\in E(G)$, hence $x_i-z^{r_i-r_j}x_j=0$ is a hyperplane in $\mathcal{A}_{r-1}$ containing the intersection of both hyperplanes, satisfying the definition of filtration of a supersolvable arrangement.
\end{proof}

\begin{proof}[Proof of Theorem \ref{chordality is freeness}]

According to Theorem \ref{filtration}, supersolvability implies freeness of an arrangement, we proved that chordality implies freeness.

Conversely, if $\mathcal{M}(G,r)$ is free but $G$ is not chordal, we can choose a minimal cycle without chord, say $C=\{v_1,v_2,\cdots,v_i\}$, where $i\geq 4$. Consider the localization of $\mathcal{A}$ on $X=\{x_1-x_2=0, x_2-x_3=0, \cdots, x_{i-1}-x_i=0, x_i-x_1=0\}$. The localization $\mathcal{A}_X=\{H\in\mathcal{A}\mid \bigcap_{H'\in X} H'\subseteq H\}$ is exactly $X$, which is isomorphic to the graphic arrangement of cycle graph $C_i$. By Theorem \ref{graphical arrangement freeness}, $\mathcal{A}_X$ cannot be free since $C_i$ is not chordal. Following the fact in \cite{orlik1992arrangements}, the localization of a free arrangement is also free, hence $\mathcal{M}(G,r)$ cannot be free, which is contradictory to the assumption. According to the previous statement, we proved Theorem \ref{chordality is freeness}.
\end{proof}

Suyama-Tsujie \cite{suyama2019vertex-weighted-dcg} and Nian-Tsujie-Uchiumi-Yoshinaga \cite{nian2024qdeformationchromaticpolynomialsgraphical} gave the bases for graphic arrangements of chordal graphs and their $q$-deformations respectively. We give a free basis for graphic monomial arrangement of chordal graph $G$ following the method in both theses. 

Let $G$ be a chordal graph with a perfect elimination ordering $(v_{1}, \dots, v_{\ell})$ and $(x_{1}, \dots, x_{\ell})$ the corresponding coordinates. 
Define the sets $C_{\geq k}$ and $E_{<k}$ by 
\begin{align*}
C_{\geq k} &\coloneqq \{k\} \cup \Set{i \in [\ell] | \exists\text{ a path } v_{k}v_{j_{1}}\cdots v_{j_{n}}v_{i} \text{ such that } k < j_{1} < \dots < j_{n} < i}, \\
E_{<k} &\coloneqq \Set{j \in [\ell] | j < k \text{ and } \{v_{j},v_{k}\} \in E_{G}}. 
\end{align*}

Let $\Delta(x_{1}, \dots, x_{k})$ denote the Vandermonde determinant: 
\begin{align*}
\Delta(x_{1}, \dots, x_{k}) \coloneqq \begin{vmatrix}
1 & x_{1} & x_{1}^{2} & \dots & x_{1}^{k-1} \\
1 & x_{2} & x_{2}^{2} &  \dots & x_{2}^{k-1} \\
\vdots & \vdots & \vdots &  & \vdots \\
1 & x_{k} & x_{k}^{2} & \dots & x_{k}^{k-1}
\end{vmatrix} 
= \prod_{1 \leq i < j \leq k}(x_{j}-x_{i}). 
\end{align*}

When $E_{<k} = \{j_{1}, \dots, j_{m}\}$ with $j_{1} < \dots < j_{m}$, 
\begin{align*}
\Delta(E_{<k}) \coloneqq \Delta(x_{j_{1}}, \dots, x_{j_{m}}) 
\quad \text{ and } \quad
\Delta(E_{<k}, x_{i}) \coloneqq \Delta(x_{j_{1}}, \dots, x_{j_{m}}, x_{i}). 
\end{align*}

Let $\Delta_{q}(x_{1}, \dots, x_{k}) \in \mathbb{F}_{q}[x_{1}, \dots, x_{k}]$ denote the determinant of the Moore matrix. 
Namely 
\begin{align*}
\Delta_{q}(x_{1}, \dots, x_{k}) &= \begin{vmatrix}
x_{1} & x_{1}^{q} & x_{1}^{q^{2}} & \dots & x_{1}^{q^{k-1}} \\
x_{2} & x_{2}^{q} & x_{2}^{q^{2}} &  \dots & x_{2}^{q^{k-1}} \\
\vdots & \vdots & \vdots &  & \vdots \\
x_{k} & x_{k}^{q} & x_{k}^{q^{2}} & \dots & x_{k}^{q^{k-1}}
\end{vmatrix}\\ 
&= \prod_{i = 1}^{k}\prod_{c_{1}, \dots, c_{i-1} \in \mathbb{F}_{q}}(c_{1}x_{1} + \dots + c_{i-1}x_{i-1} + x_{i}),
\end{align*}
then we have the following theorem.

\begin{theorem}[Suyama-Tsujie \cite{suyama2019vertex-weighted-dcg}, Nian-Tsujie-Uchiumi-Yoshinaga \cite{nian2024qdeformationchromaticpolynomialsgraphical}]\label{freeness1}

$$\{\theta_{k} = \sum_{i \in C_{\geq k}}\dfrac{\Delta(E_{<k}, x_{i})}{\Delta(E_{<k})}\partial_{i}\mid k=1,2,\cdots,\ell\}$$
gives a free basis of $\mathcal{A}_G$ and
$$\{\theta^q_{k} = \sum_{i \in C_{\geq k}}\dfrac{\Delta_q(E_{<k}, x_{i})}{\Delta_q(E_{<k})}\partial_{i}\mid k=1,2,\cdots,\ell\}$$
gives a free basis of $\mathcal{A}_G^q$.
\end{theorem}

% \subsection{Graphic monomial arrangements}

As for graphic monomial arrangement, we follow the method in both theses. Let $\Delta^{(r)}_1(x_{1}, \dots, x_{k}) \in \mathbb{K}[x_{1}, \dots, x_{k}]$ denote the determinant of the following matrix.

\begin{align*}
\Delta^{(r)}_1(x_{1}, \dots, x_{k}) &= \begin{vmatrix}
x_{1} & x_{1}^{r+1} & x_{1}^{2r+1} & \dots & x_{1}^{(k-1)r+1} \\
x_{2} & x_{2}^{r+1} & x_{2}^{2r+1} & \dots & x_{2}^{(k-1)r+1} \\
\vdots & \vdots & \vdots &  & \vdots \\
x_{k} & x_{k}^{r+1} & x_{k}^{2r+1} & \dots & x_{k}^{(k-1)r+1}
\end{vmatrix}\\ 
&= \prod_{i=1}^k x_i \prod_{1\leq i < j \leq k} (x_j^{r}-x_i^{r}).\\ 
\end{align*}
Following the same notations, we have the following theorem.
\begin{theorem}
    Let $G$ be a chordal graph with perfect elimination ordering $(v_1,v_2,\cdots,v_\ell)$, $z$ be a primitive $r$-root in $\mathbb{K}$, then $\mathcal{M}(G,r)$ is a free arrangement and
    $$\{\theta_{k} = \sum_{i \in C_{\geq k}}\dfrac{\Delta^{(r)}_1(E_{<k}, x_{i})}{\Delta^{(r)}_1(E_{<k})}\partial_{i}\mid k=1,2,\cdots,\ell\}$$
where $\Delta_1(E_{<1})=1$ gives a free basis of $\mathcal{M}(G,r)$. 

As a result, if the exponents of $\mathcal{A}_G$ is $(e_1,e_2,\cdots,e_\ell)$, the exponents of $\mathcal{M}(G,r)$ is $(re_1+1,re_2+1,\cdots,re_\ell+1)$.
\end{theorem}

\begin{proof}
    Let $e=\{v_i,v_j\}$ be an edge of $G$ with $i < j$. There are two types of hyperplanes in $\mathcal{M}(G,r)$. One is the kernel of $\alpha=x_i$ and it is easy to check in this case $\theta_k(\alpha)\in(\alpha)$. 

The other type is the kernel of $\alpha = x_i - z^ax_j$. If $e \cap C_{\geq k} = \varnothing$, then $\theta_{k}(\alpha) = 0$. 

Suppose that $e \cap C_{\geq k} \neq \varnothing$. There are only two cases, $e \cap C_{\geq k} =\{v_j\}$ and $e \cap C_{\geq k} =\{v_i,v_j\}$ since $v_i$ and $v_j$ are adjacent so if there is an ascending path from $v_k$ to $v_i$, then there is also an ascending path to $v_j$ by adding one node.

Note that if $e \cap C_{\geq k} =\{v_j\}$, then $i\in E_{<k}$ since there is an ascending path $v_kv_{k_1}v_{k_2}\dots,v_{k_n}v_j$ in the perfect elimination ordering, hence $v_i$ and $v_{k_n}$ are adjacent. Hence $i<k_n$ since otherwise we can find a new ascending path containing $i$, contradicting the assumption. By repeating this argument, 

we have $i<k$ and $v_i$ is adjacent to $v_k$, i.e., $i\in E_{<k}$.

$$\theta_{k}(\alpha) = -z^a\dfrac{\Delta^{(r)}_1(E_{<k}, x_{j})}{\Delta^{(r)}_1(E_{<k})} = -z^ax_j \prod_{m\in E_{<k}}(x_j^{r}-x_m^{r})$$

one can show that $\theta_{k}(\alpha)$ is a multiple of $\alpha$ since  $x_i-z^ax_j$ is a factor of $x_j^{r}-x_i^{r}$ and also a factor of $\theta_{k}(\alpha)$ by $i\in E_{<k}$.

The last case is $e \cap C_{\geq k} =\{v_i,v_j\}$, by computation,

\begin{align*}
    \theta_{k}(\alpha) &= -z^a\dfrac{\Delta^{(r)}_1(E_{<k}, x_{j})}{\Delta^{(r)}_1(E_{<k})}+\dfrac{\Delta^{(r)}_1(E_{<k}, x_{i})}{\Delta^{(r)}_1(E_{<k})}\\
    & \equiv 0 \pmod {x_i-z^ax_j}
\end{align*}

by letting $x_i=z^ax_j$, then one can show that $\alpha=x_i-z^ax_j$ is a factor of $\theta_k(\alpha)$.

By previous arguments, one can show that $\theta_1, \theta_2,\cdots,\theta_\ell$ are all derivations in $D(\mathcal{M}(G,r))$.

To check if $\theta_1, \theta_2,\cdots,\theta_\ell$ give a free basis of $D(\mathcal{M}(G,r))$, we use Saito's criterion \cite{saito1980theory}, the Saito's matrix $(\theta_i x_j)_{i,j}$ is upper triangular since $\theta_i x_j=0$ for any pair $i>j$. The determinant of Saito's matrix is the product of diagonal entries, i.e., $\prod_{i=1}^\ell \theta_i x_i$, a direct computation shows

\begin{align*}
    \prod_{i=1}^\ell \theta_i x_i &= \prod_{i=1}^\ell \dfrac{\Delta^{(r)}_1(E_{<i}, x_{i})}{\Delta^{(r)}_1(E_{<i})}\partial_{i} x_i\\
    &= \prod_{i=1}^\ell \dfrac{\Delta^{(r)}_1(E_{<i}, x_{i})}{\Delta^{(r)}_1(E_{<i})}\\
    &= \prod_{i=1}^\ell (x_i\prod_{j\in E_{<i}} (x_i^{r}-x_j^{r}))\\
    &= \prod_{i=1}^\ell x_i \prod_{\{i,j\}\in E(G)}(x_i^{r}-x_j^{r})\\
    &= Q(\mathcal{M}(G,r))    
\end{align*}
since $x_i^{r}-x_j^{r}=\prod_{k=0}^{r-1}(x_i-z^k x_j)$ in $\mathbb{K}$. Hence $\{\theta_1, \theta_2,\cdots,\theta_\ell\}$ is a free basis. 
\end{proof}

\section{Acknowledgement}
The author would like to express sincere gratitude to the author's academic adviser Professor Masahiko Yoshinaga for proposing this stimulating problem and for the insightful guidance throughout this work. The author is also grateful to Professor Takuro Abe and Professor Paul M{\"u}cksch for the constructive discussions we shared while working on this problem.

\printbibliography
\end{document}